\documentclass[11pt, a4paper]{amsart}

\usepackage[ansinew]{inputenc}
\usepackage[all]{xy}
\SelectTips{cm}{}
\usepackage[pagebackref,
  ,pdfborder={0 0 0}
  ,urlcolor=black,a4paper,hypertexnames=false]{hyperref}
\hypersetup{pdfauthor={Clara L\"oh, Marco Moraschini}
  ,pdftitle=Topological volumes of fibrations: A note on open covers}

\usepackage{amsfonts}
\usepackage{amssymb, amsmath,amsthm, mathtools}
\usepackage{mathrsfs}  
\usepackage{tabularx, hyperref}
\usepackage{enumerate}

\usepackage{tikz-cd}

\makeatletter
\newtheorem*{rep@theorem}{\rep@title}
\newcommand{\newreptheorem}[2]{%
\newenvironment{rep#1}[1]{%
 \def\rep@title{#2 \ref{##1}}%
 \begin{rep@theorem}}%
 {\end{rep@theorem}}}
\makeatother

\newtheorem{lemma}{Lemma}[section]
\newtheorem{thm}[lemma]{Theorem}
\newreptheorem{thm}{Theorem}  
\newtheorem{prop}[lemma]{Proposition}
\newreptheorem{prop}{Proposition}  
\newtheorem{cor}[lemma]{Corollary}
\newreptheorem{cor}{Corollary}  

\theoremstyle{definition}
\newtheorem{defi}[lemma]{Definition}

\newtheorem{quest}[lemma]{Question}
\newreptheorem{quest}{Question}  

\newtheorem{example}[lemma]{Example}
\newtheorem{rem}[lemma]{Remark}

\theoremstyle{definition}

\makeatletter
\newcommand\norm{\bBigg@{0.8}}

\makeatother
 
 \newcommand{\indnorm}[2][flex]{\csname #1l\endcsname\|#2%
                                 \csname #1r\endcsname\|\mathclose{}}
                                  \newcommand{\indnorml}[4][flex]{\csname #1l\endcsname\|#2%
                                 \csname #1r\endcsname\|_{#3}^{#4}\mathclose{}}
\newcommand{\sv}[2][flex]{\indnorm[#1]{#2}}

\DeclareMathOperator{\catop}{cat}
\DeclareMathOperator{\LS}{LS}
\DeclareMathOperator{\Am}{Am}
\DeclareMathOperator{\amcat}{\catop_{\Am}}
\DeclareMathOperator{\lscat}{\catop_{\LS}}

\def\gcat{\catop_{\mathscr{G}}}

\DeclareMathOperator{\Poly}{Poly}
\DeclareMathOperator{\Subexp}{Subexp}
\DeclareMathOperator{\Exp}{Exp}
\DeclareMathOperator{\fingen}{fg}
\def\Polyfg{\Poly^{\fingen}}
\def\Subexpfg{\Subexp^{\fingen}}
\def\Expfg{\Exp^{\fingen}}
\DeclareMathOperator{\uexp}{uexp}

\DeclareMathOperator{\Ob}{Ob}
\DeclareMathOperator{\Group}{{\sf Group}}

\DeclareMathOperator{\mult}{mult}
\DeclareMathOperator{\ent}{minent}
\DeclareMathOperator{\volent}{ent}
\DeclareMathOperator{\Riem}{Riem}
\DeclareMathOperator{\vol}{vol}

\DeclareMathOperator{\comp}{comp}

\DeclareMathOperator{\im}{im}

\DeclareMathOperator{\id}{id}

\newcommand{\fa}[1]{%
  \forall_{#1}\quad}

\newcommand{\qand}{%
  \qquad\text{and}\qquad}

\newcommand{\N}{\ensuremath {\mathbb{N}}}
\newcommand{\R} {\ensuremath {\mathbb{R}}}

\newcommand{\Z} {\ensuremath {\mathbb{Z}}}

\renewcommand{\rho}{\varrho}
\def\phi{\varphi}

\def\args{\;\cdot\;}
\DeclareMathOperator{\const}{const}

\usepackage{color}
\usepackage{pdfcolmk}

\long\def\forget#1{}

\def\longrightarrow{\rightarrow}
\def\longmapsto{\mapsto}
\def\widetilde{\tilde}

\makeatletter
\@namedef{subjclassname@2020}{%
  \textup{2020} Mathematics Subject Classification}
\makeatother

\begin{document}

\title[Topological volumes of fibrations: A note on open covers]%
      {Topological volumes of fibrations:\\ A note on open covers}

\author{Clara L\"{o}h}
\address{Fakult\"{a}t f\"{u}r Mathematik, Universit\"{a}t Regensburg, Regensburg, Germany}
\email{clara.loeh@ur.de}

\author{Marco Moraschini}
\address{Fakult\"{a}t f\"{u}r Mathematik, Universit\"{a}t Regensburg, Regensburg, Germany}
\email{marco.moraschini@ur.de}


\date{\today.\ \copyright{\ C.~L\"oh, M.~Moraschini}.
  This work was supported by the CRC~1085 \emph{Higher Invariants}
  (Universit\"at Regensburg, funded by the~DFG)}

\begin{abstract}
  We establish a straightforward estimate for the number of open sets
  with fundamental group constraints needed to cover the total space of
  fibrations. This leads to vanishing results for simplicial volume
  and minimal volume entropy, e.g., for certain mapping tori.
\end{abstract}

\maketitle

\section{Introduction}

The Lusternik-Schnirelmann category~$\lscat(X)$ of a topological
space~$X$ is the minimal number of open and in~$X$ contractible sets
necessary to cover~$X$. Despite the first applications of Lusternik-Schnirelmann category
were more focused on the study of critical points~\cite{LS-original, CorneaCritical}, 
it is now widely applied also to the study of algorithms' complexity~\cite{SmaleAlg, Vassiliev}
and topological robotics~\cite{Farber:invitation}.

Relaxing the contractibility condition leads to generalised categorical 
invariants~$\catop_\mathscr{G}$ with fundamental group constraints
(Section~\ref{sec:LS}): Let $\mathscr{G}$ be a class of groups. A \emph{$\mathscr{G}$-set} in
a space~$X$ is a subset whose path-connected components all have
$\pi_1$-image in~$\mathscr{G}$. The $\mathscr{G}$-category of~$X$ is the minimal
number~$\catop_\mathscr{G}(X)$ of open $\mathscr{G}$-sets needed to cover~$X$. 
Geometrically relevant classes~$\mathscr{G}$
are the class~$\Am$ of amenable groups or classes
of groups with controlled growth (for instance, $\Subexp_{<\delta}$).

The main observation of the present note is the following
straightforward adaption of an estimate for the
Lusternik-Schni\-rel\-mann category~\cite{varadarajan, Hardie70, James} to the
case of fundamental group constraints:

\begin{thm}[$\mathscr{G}$-category of fibrations]\label{thm:main}
  Let $p\colon E \longrightarrow B$ be a fibration with a
  path-connected base space. Let $x_0 \in B$ be a non-degenerate
  basepoint of~$B$ and let $F := p^{-1}(x_0)$ denote the fibre.  Moreover, let
  $\mathscr{G}$ be a class of groups that is closed under isomorphisms, 
  subgroups, and quotients. Then:
  \[ \catop_\mathscr{G} (E) \leq \catop_\mathscr{G} (F) \cdot \lscat (B)
  \]
\end{thm}

We prove this statement in Theorem~\ref{thm:G:cover:fibrations} in
terms of categorical invariants of maps. It should be noted that the
rough estimate provided by Theorem~\ref{thm:main}, in general, cannot
be improved to~$\catop_\mathscr{G} (E) \leq \catop_\mathscr{G} (F) \cdot \catop_\mathscr{G} (B)$
(Example~\ref{rem:genmain}).

\subsection*{Applications to topological volumes}

Simplicial volume and minimal volume entropy are examples of
``topological volumes'', i.e., of $\R$-valued invariants of manifolds
that mitigate between topological properties and Riemannian volume
(Section~\ref{subsec:simvol}, Section~\ref{subsec:minvolent}).  Both
simplicial volume and minimal volume entropy admit vanishing theorems
in terms of $\pi_1$-constrained open covers of small
multiplicity. Therefore, Theorem~\ref{thm:main} gives corresponding
vanishing results for fibrations and fibre bundles.  For example:

For simplicial volume, we can combine Theorem~\ref{thm:main} with
Gromov's vanishing theorem for bounded cohomology and amenable
covers~\cite{vbc} (Theorem~\ref{grom:van:thm}):

\begin{cor}[Corollary~\ref{cor:G:cover:fibrations:vanishing:simvol}]\label{cor:mainsv}
  Let $M$ be an oriented closed connected manifold that is
  the total space of a fibre bundle~$M \to B$ with oriented
  closed connected fibre~$N$ and base~$B$. If
  \[ \amcat (N) \leq \frac{\dim (M)}{\dim(B) + 1},
  \]
  then~$\|M\| = 0$.
\end{cor}

The fibre collapsing assumption of Babenko and Sabourau can be
translated into generalised Lusternik-Schnirelmann category invariants
(Section~\ref{sec:fnca}). Therefore, combining Theorem~\ref{thm:main}
with the vanishing result of Babenko and
Sabourau~\cite[Theorem~1.3]{bsfibre} (Theorem~\ref{thm:entfca}), we
obtain:

\begin{cor}[Corollary~\ref{cor:G:cover:fibrations:vanishing:minvolent}]\label{cor:mainminvolent}
  Let $M$ be an oriented closed connected smooth manifold that is
  the total space of a fibre bundle~$M \to B$ with oriented
  closed connected smooth fibre~$N$ and base~$B$. If
  \[ \catop_{\Subexp_{<1/\dim(M)}} (N) \leq \frac{\dim (M)}{\dim(B) + 1},
  \]
  then~$\ent(M) = 0$.
\end{cor}

In particular, Corollary~\ref{cor:mainsv} and
Corollary~\ref{cor:mainminvolent} lead to vanishing results for
certain mapping tori. In the case of simplicial volume, this
complements vanishing results of Bucher and
Neofytidis~\cite{bucherneofytidis} (Remark~\ref{rem:bn}).

Moreover, we use generalised Lusternik-Schnirelmann category
invariants to generalise a result by Bregman and
Clay~\cite[Proposition~4.1]{bregmanclay} on the fibre collapsing
assumptions and graphs of groups (Remark~\ref{rem:faccat2}) and
provide an aspherical version of examples of Babenko and
Sabourau~\cite[Theorem~1.6]{bsfibre} of simplicial complexes with
large minimal volume entropy and small ``simplicial volume''
(Proposition~\ref{prop:fncasimvol}, Corollary~\ref{cor:entsimvol}).

\subsection*{Organisation of this article}

We recall the generalised Lusternik-Schni\-rel\-mann category in
Section~\ref{sec:LS}.  The proof of Theorem~\ref{thm:main} is given in
Section~\ref{sec:fib}; applications to simplicial volume are
contained in Section~\ref{sec:appsv}.  In Section~\ref{sec:fnca}, we
recall the fibre (non-)collapsing assumption by Babenko and Sabourau;
finally, the applications to minimal volume entropy are located in
Section~\ref{sec:appminvolent}.

\subsection*{Acknowledgements}

We would like to thank George Raptis, Pietro Capo\-vil\-la, and Kevin Li for
carefully reading a previous version.

\section{Generalised LS-category}\label{sec:LS}

We recall the definition of the Lusternik-Schnirelmann category (of
spaces and maps) and the generalisation to fundamental group
constraints~\cite{Fox,James,CLOT,eilenbergganea,GGH,clm}.

\subsection{LS-Category of spaces and maps}

\begin{defi}[$\LS$-category]
  Let $X$ be a topological space. The \emph{Lusternik-Schnirelmann
    category} (or simply $\LS$-\emph{category}) of $X$, denoted
  by~$\lscat(X)$, is the minimal number~$n \in \, \N = \{0,1,\dots\}$ such that $X$
  can be covered with open sets~$U_1, \dots, U_n$ that are
  contractible in~$X$.  If such an~$n$ does not exist, we set
  $\lscat(X) := +\infty$.
\end{defi}

Similarly, we have the definition of $\LS$-category of a continuous map:

\begin{defi}[$\LS$-category of a map]
  Let $f \colon X \to Y$ be a continuous map between topological
  spaces.  The $\LS$-\emph{category} of~$f$, denoted by~$\lscat(f)$,
  is the minimal number~$n \in \, \N$ such that $X$ can be covered
  with open sets~$U_1, \dots, U_n$ such that the restriction~$f
  |_{U_i}$ is null-homotopic for each~$i \in \{1,\dots, n\}$. If such
  a number~$n$ does not exist, we set $\lscat(f) := +\infty$.
\end{defi}

\subsection{$\mathscr{G}$-Category of spaces and maps}

\begin{defi}[$\mathscr{G}$-sets (for a map)]
  Let $\mathscr{G}$ be a class of groups.
  Let $X$ be a topological space and let $U$ be a subset of~$X$. 
  We say that $U$ is a \emph{$\mathscr{G}$-set} if for all~$x \in \, U$, 
  we have
  $$
  \im \bigl( \pi_1(U \hookrightarrow X, x) \bigr) \in \, \mathscr{G}.
  $$
  An open cover $\mathcal{U}$ of $X$ is called a $\mathscr{G}$-\emph{cover} if
  each open subset in~$\mathcal{U}$ is a $\mathscr{G}$-set.

  Similarly, given a continuous map $f \colon X \to Y$ between topological spaces,
  we say that an open set~$U \subset X$ is a \emph{$\mathscr{G}$-set for~$f$} if 
  for every $x \in \, U$, we have
  $$
  \pi_1(f) \bigl( \pi_1(U, x) \bigr) \in \, \mathscr{G}.
  $$
  An open cover~$\mathcal{U}$ of~$X$ is called a $\mathscr{G}$-\emph{cover
    for}~$f$ if each open subset in~$\mathcal{U}$ is a $\mathscr{G}$-set
  for~$f$.
\end{defi}

\begin{defi}[$\mathscr{G}$-category (of a map)]
  Let $\mathscr{G}$ be a class of groups.
  Let $X$ be a topological space. The $\mathscr{G}$-\emph{category} of~$X$,
  denoted by $\catop_\mathscr{G}(X)$, is the minimal number~$n \in \, \N$ such that
  $X$ admits an open $\mathscr{G}$-cover of cardinality~$n$.  If such an~$n$ does
  not exist, we set $\catop_\mathscr{G}(X) := +\infty$.

  Similarly, the $\mathscr{G}$-\emph{category of a continuous map}~$f \colon X
  \to Y$ between topological spaces, denoted by $\catop_\mathscr{G}(f)$, is the
  minimal number~$n \in \, \mathbb{N}$ such that $X$ admits an open
  $\mathscr{G}$-cover for~$f$ of cardinality~$n$. If such an~$n$ does
  not exist, we set $\catop_\mathscr{G}(f) := +\infty$.
\end{defi}

\begin{rem}
  Let $\mathscr{G}$ be a class of groups. If $\mathscr{H}$ is a class of groups
  with~$\mathscr{H} \subset \mathscr{G}$, then $\catop_\mathscr{G} \leq \catop_\mathscr{H}$. 
  In particular: If $\mathscr{G}$ contains all trivial groups and $X$ is
  a finite-dimensional simplicial complex, then
  $\catop_\mathscr{G} (X) \leq \dim (X) + 1,
  $ 
  as can be seen by the open stars cover of the barycentric
  subdivision of~$X$ (grouped and indexed by the simplices of~$X$). 
\end{rem}

\begin{rem}\label{rem:homotopic:map:same:G:cat}
  We say that topological spaces $X$ and~$Y$ are $\pi_1$-\emph{equivalent}, if
  there exists continuous maps (called \emph{$\pi_1$-equivalences}) 
  $X \to Y$ and $Y \to X$ inducing isomorphisms on
  the level of fundamental groups (these maps are not required to
  be $\pi_1$-inverse to each other).
  
  Similarly, maps $f, g \colon X \to Y$ are said to be $\pi_1$-\emph{homotopic}
  if they induce the same homomorphism on the level of fundamental groups.

  Let $\mathscr{G}$ be a class of groups that is closed under isomorphism and
  let $X, Y$ be $\pi_1$-equivalent spaces. Then, $\gcat(X) =
  \gcat(Y)$. 
  Similarly, if $f, g \colon X \to Y$ are $\pi_1$-homotopic
  maps, then $\gcat(f) = \gcat(g)$.
\end{rem}

We collect some basic properties of the $\mathscr{G}$-category of a map,
which are known to hold in the case of $\LS$-category~\cite[Exercise~1.16]{CLOT}.

\begin{lemma}[properties of $\mathscr{G}$-category]\label{lemma:properties:G:cat:maps}
  Let $\mathscr{G}$ be a class of groups that is closed under isomorphisms, subgroups,
  and quotients. 
  Let $f \colon X \to Y$ and $g \colon Y \to Z$ be continuous maps
  between topological spaces. Then, we have the following:
  \begin{enumerate}
  \item $\gcat(f) \leq \min \{\gcat(X), \gcat(Y)\}$;
  \item $\gcat(g \circ f) \leq \min\{\gcat(g), \gcat(f)\}$;
  \item If $f$ is a homotopy equivalence, then $\gcat(f) = \gcat(X) = \gcat(Y)$.
\end{enumerate}
\end{lemma}

\begin{proof}
\emph{Ad.~1.} Let $\mathcal{U}$ be an open $\mathscr{G}$-cover of $X$.
Then, for each $U \in \, \mathcal{U}$ and every~$x \in \, U$, we have
the following commutative diagram
\[
\xymatrix{
\pi_1(U, x) \ar[d] \ar[rr]^-{\pi_1(f |_U)} && \pi_1\bigl(Y, f(x)\bigr)  \ar@{=}[d] \\
\pi_1(X, x) \ar[rr]_-{\pi_1(f)} && \pi_1\bigl(Y, f(x)\bigr),
}
\]
where the left vertical arrow is induced by the inclusion.  Because
$U$ is a $\mathscr{G}$-set in~$X$ and $\mathscr{G}$ is closed under quotients, the
previous diagram shows that $U$ is a $\mathscr{G}$-set for~$f$. Therefore,
$\mathcal{U}$ is an open $\mathscr{G}$-cover for~$f$.  Taking the infimum over
all open $\mathscr{G}$-covers of~$X$ shows that $\gcat(f) \leq \gcat(X)$.

On the other hand, if $\mathcal{V}$ is an open $\mathscr{G}$-cover of $Y$, then
$(f^{-1}(V))_{V \in \, \mathcal{V}}$ is an open $\mathscr{G}$-cover
for~$f$. Hence, we get $\gcat(f) \leq \gcat(Y)$.

\emph{Ad.~2.} As $\mathscr{G}$ is closed under taking quotients, it is immediate
to check that $\gcat(g \circ f) \leq \gcat(f)$.

Moreover, if $\mathcal{U}$ is an open $\mathscr{G}$-cover for $g$ of~$Y$, we can consider 
the pullback $f^{-1}\mathcal{U}$, which is an open $\mathscr{G}$-cover for~$g \circ f$.
Therefore, $\gcat(g \circ f ) \leq \gcat(g)$.

\emph{Ad.~3.} Because $X$ and $Y$ are homotopy equivalent, we have
$\gcat(X) = \gcat(Y)$ (Remark~\ref{rem:homotopic:map:same:G:cat}). 
Let $f \colon X \to Y$ be a homotopy equivalence and let $g \colon Y
\to X$ be a homotopy inverse. Then, $g \circ f$ is homotopic to
$\id_X$. Then Remark~\ref{rem:homotopic:map:same:G:cat} and the first
two parts show that
\begin{align*}
  \gcat(X)
  & = \gcat(\id_X) = \gcat(g \circ f) \\
  & \leq \min\{\gcat(f), \gcat(g)\}
  \leq \max\{\gcat(f), \gcat(g)\}
  \leq \gcat(X).
\end{align*}
This shows that $\gcat(Y) = \gcat(X) = \gcat(f)$.
\end{proof}

\section{Generalised LS-category and fibrations}\label{sec:fib}

In this section, we prove Theorem~\ref{thm:main}. Indeed,
Lemma~\ref{lemma:properties:G:cat:maps} shows that
Theorem~\ref{thm:main} is a consequence of the following statement:

\begin{thm}[$\mathscr{G}$-category of fibrations]\label{thm:G:cover:fibrations}
  Let $p\colon E \longrightarrow B$ be a fibration with a
  path-connected base space. Let $x_0 \in B$ be a non-degenerate
  basepoint of~$B$, let $F := p^{-1}(x_0)$, and let $\iota \colon F
  \hookrightarrow E$ denote the inclusion of the fibre.  Moreover, let
  $\mathscr{G}$ be a class of groups that is closed under isomorphisms and
  subgroups. Then:
  \[ \catop_\mathscr{G} (E) \leq \catop_\mathscr{G} (\iota) \cdot \lscat (p)
  \]
\end{thm}

Recall that a basepoint~$x_0$ of a space~$B$ is \emph{non-degenerate}
if the inclusion~$\{x_0\} \to B$ is a cofibration.

\begin{proof} 
  Let $n := \catop_\mathscr{G} (\iota)$ and $b := \lscat (p)$; without loss of
  generality, we may assume that they are both finite.
  Let $(V_i)_{i \in [n]}$ and $(W_j)_{j \in [b]}$ be corresponding open
  covers of~$F$ and $E$, respectively; here, for~$k \in \N$, we
  abbreviate~$[k] := \{1,\dots,k\}$. 
  We construct an open $\mathscr{G}$-cover~$(U_{ij})_{i\in [n],j\in[b]}$ of~$E$
  as follows:

  Let $j \in [b]$. As $W_j$ is an LS-set for~$p$
  (i.e., $p|_{W_j} \colon W_j \to B$ is null-homotopic) and
  as the basepoint~$x_0 \in B$ is non-degenerate, there exists
  a homotopy~$h_j \colon W_j \times [0,1] \longrightarrow B$
  with~$h_j(\args,0) = p|_{W_j}$ and $h_j(\args,1) = \const_{x_0}$.
  By the homotopy lifting property, there exists a
  homotopy~$\widetilde h_j \colon W_j \times [0,1] \longrightarrow E$
  with
  $p \circ \widetilde h_j = h_j.
  $
  In particular,
  \[ \widetilde h_j(\args,1) (W_j) \subset p^{-1}(x_0)= F.
  \]
  We write~$g_j := \widetilde h_j(\args,1) \colon W_j \longrightarrow F$.

  For all~$i \in [n]$ and all~$j \in [b]$, we set
  \[ U_{ij} := g_j^{-1}(V_i) \subset E. 
  \]
  By construction, $U_{ij}$ is open in~$E$ and $\bigcup_{(i,j) \in [n]
    \times [b]} U_{ij} = E$.

  It remains to show that each~$U_{ij}$ is a $\mathscr{G}$-set.  Let $U \subset
  U_{ij}$ be a path-component of~$U_{ij}$, let $i_U \colon U
  \hookrightarrow E$ be the inclusion and let $y_0 \in U$.

  We consider the map
  \[ k := \widetilde h_j |_{U \times [0,1]}
     \colon U \times [0,1] \longrightarrow E.
  \]
  Then $k(\args,0) = i_U$. Moreover, let $y_1 := k(y_0,1)$ and let $
  \alpha \colon [0,1] \longrightarrow E,\ t \longmapsto k(y_0,t)$.
  Then, we obtain the corresponding change of basepoints isomorphism
  \begin{align*}
    \alpha_* \colon \pi_1(E,y_0) & \longrightarrow \pi_1(E,y_1)\\
    [\gamma] & \longmapsto [\alpha^{-1} * \gamma * \alpha].
  \end{align*}
  By construction,
  \[ \alpha_* \circ \pi_1(i_U) = \pi_1\bigl(k(\args,1)\bigr)
  \colon \pi_1(U,y_0) \longrightarrow \pi_1(U,y_1).
  \]
  Because $\alpha_*$ is an isomorphism and $\mathscr{G}$ is closed under
  isomorphisms, it suffices to show that $\Lambda :=
  \pi_1(k(\args,1))(\pi_1(U,y_0))$ lies in~$\mathscr{G}$.  The commutative
  diagram
  \[ \xymatrix{%
    U \ar[r]^-{k(\args,1)} \ar[d]_{\txt{incl}}
    &
    E
    \\
    U_{ij} \ar[r]_{g_j|_{U_{ij}}}
    & F \ar[u]_-{\iota}
    }
  \]
  shows that
  $\Lambda$ is a subgroup of~$\pi_1(\iota) (\pi_1(V_i, g_j(y_0))$. The latter
  group is in~$\mathscr{G}$ as $V_i$ is a $\mathscr{G}$-set for~$\iota$. Because $\mathscr{G}$ is closed
  under subgroups, we obtain~$\Lambda \in \mathscr{G}$.

  Therefore, $(U_{ij})_{i \in [n], j \in [b]}$ is an open $\mathscr{G}$-cover of~$E$
  and so~$\catop_\mathscr{G} E \leq n \cdot b$.
\end{proof}

\begin{rem}\label{rem:equal:cat:cat:amenable}
  Let $p \colon E \to B$ be a fibration with fibre~$F$
  over the basepoint~$x_0 \in B$ and let $i \colon F \hookrightarrow E$
  be the inclusion. 
  Let $\mathscr{G}$ be a class of groups that is closed under isomorphisms,
  subgroups, quotients, and under extensions by Abelian kernels; e.g., $\mathscr{G} = \Am$.
  Then $\catop_\mathscr{G} (i) = \catop_\mathscr{G}(F)$:
  By Lemma~\ref{lemma:properties:G:cat:maps}, we already know that
  $\catop_\mathscr{G} (i) \leq \catop_\mathscr{G}(F)$. On the other hand, the long exact
  sequence for~$p$ and the closure properties of~$\mathscr{G}$ show that
  $\catop_\mathscr{G}(F) \leq \catop_\mathscr{G}(i)$.
\end{rem}

\begin{cor}[$\mathscr{G}$-category of mapping tori]\label{cor:vanishing:mapping:tori:G}
  Let $\mathscr{G}$ be a class of groups that is closed under isomorphisms and subgroups.
  Let $M$ be a closed connected manifold that admits a fibre bundle~$p \colon
  M\longrightarrow S^1$ with manifold fibre~$i \colon N \hookrightarrow M$. If
  $2 \cdot \catop_\mathscr{G} (i) \leq \dim N +1$, then
  \[ \catop_\mathscr{G} (M) \leq \dim M.
  \]
\end{cor}
 \begin{proof}
   Because $\lscat(S^1) = 2$, we have $\lscat(p) \leq 2$. 
   Applying Theorem~\ref{thm:G:cover:fibrations} and the hypothesis
   on the $\mathscr{G}$-category of~$i$, we obtain
   \[
   \gcat(M)
   \leq  \lscat(p) \cdot \catop_\mathscr{G}(i)
   \leq 2 \cdot \catop_\mathscr{G}(i)
   \leq \dim(N) + 1
   = \dim(M).
   \qedhere
   \]
\end{proof}

\section{Applications to simplicial volume}\label{sec:appsv}

We recall the definition of simplicial volume and bounded cohomology
(Section~\ref{subsec:simvol}) and Gromov's vanishing theorem (Section~\ref{subsec:simvolvan}).
In Section~\ref{subsec:simvolfib}, we derive the vanishing results for
simplicial volume and fibrations.

\subsection{Bounded cohomology and simplicial volume}\label{subsec:simvol}

We briefly recall the definition of bounded cohomology for spaces and
of simplicial volume. The systematic use of these invariants in
geometry was initiated by Gromov~\cite{vbc}. Simplicial volume
measures manifolds in terms of singular chains.

Given a topological space~$X$, we denote the real singular chain complex
by~$(C_\bullet(X; \R), \partial_\bullet)$ and the real singular cochain complex
by~$(C^\bullet(X; \R), \delta^\bullet)$. These complexes are endowed with
norms:
For a singular $n$-chain~$c = \sum_{i = 1}^k \alpha_i \cdot \sigma_i
\in \, C^n(X; \R)$ in reduced form, we define the $\ell^1$-\emph{norm} by:
\[
  |c|_1 := \sum_{i = 1}^k |\alpha_i|.
\]
Similarly, we endow~$C^\bullet(X; \R)$ with the $\ell^\infty$-\emph{norm}
given by 
\[
  |\varphi|_\infty
  := \sup \bigl\{ |\varphi(\sigma)| \bigm| \sigma \mbox{ is a singular $n$-simplex in $X$}\bigr\}
  \in \R_{\geq 0 } \cup \{\infty\}
\]
for all $\varphi \in \, C^n(X; \R)$. As the coboundary
operator~$\delta^\bullet$ maps cochains of finite norm to cochains of
finite norm, we obtain the subcomplex~$(C_b^\bullet(X; \R), \delta^\bullet)$
of singular cochains~$\varphi$ with~$|\varphi|_\infty < \infty$.

\begin{defi}[simplicial volume]
  Let $M$ be an oriented closed connected $n$-dimensional manifold. We
  define the \emph{simplicial volume} of~$M$ to be
  \[
  \sv{M} := 
  \inf \bigl\{ |c|_1 \bigm| \mbox{$c \in C_n(M;\R)$ is a cycle representing $[M]$} \bigr\},
  \]
  where $[M] \in \, H_n(M; \R)$ denotes the fundamental class of~$M$.
\end{defi}

\begin{rem}
  More generally, the $\ell^1$-norm on the singular chain complex
  induces a semi-norm~$\|\args\|_1$ on the whole singular homology
  with $\R$-coefficients.
\end{rem}

A useful tool for detecting positivity/vanishing of
simplicial volume is bounded cohomology:

\begin{defi}[bounded cohomology]
  Bounded cohomology of spaces is the functor~$H_b^\bullet(\args;\R)
  := H^\bullet(C^\bullet_b(\args;\R))$.
\end{defi}

The connection between bounded cohomology and simplicial volume
is expressed in terms of the \emph{comparison map} 
$
\comp^\bullet \colon H^\bullet_b(\args; \R) \to H^\bullet(\args;\R),
$ 
induced by the inclusion~$C_b^\bullet(\args;\R) \hookrightarrow C^\bullet(\args;\R)$.

\begin{prop}[duality principle]\label{prop:duality:principle}
  Let $M$ be an oriented closed connected $n$-manifold. 
  Then:
  \[
  \sv M > 0 \Longleftrightarrow \comp_M^n \mbox{ is surjective}
  \]
\end{prop}

\begin{rem}\label{rem:possimvol}
  By now, many examples of manifolds with non-zero simplicial volume
  are known. We list some of them:
  \begin{itemize}
  \item oriented closed connected hyperbolic manifolds~\cite{Thurston,
    vbc}; in particular, surfaces of genus~$\geq 2$;
  \item more generally: oriented closed connected, rationally
    essential manifolds of dimension~$\geq 2$ with non-elementary
    hyperbolic fundamental group~\cite[mapping
      theorem]{vbc}\cite{mineyev};
  \item oriented closed connected locally symmetric spaces of
    non-compact type~\cite{Bucher:lss, Lafont-Schmidt};
  \item manifolds with sufficiently negative curvature~\cite{inoueyano,
    connellwang};
  \item
    The class of manifolds with positive simplicial volume is closed
    with respect to connected sums and products~\cite{vbc}. 
  \end{itemize}
\end{rem}

\subsection{Gromov's vanishing theorem}\label{subsec:simvolvan}

A classical application of the duality principle
(Proposition~\ref{prop:duality:principle}) is to show that the
simplicial volume of all oriented closed connected manifolds with
amenable fundamental group (and non-zero dimension) is zero. This is a
consequence of the vanishing of the bounded cohomology for spaces with
amenable fundamental group~\cite{vbc, Ivanov, FM:Grom}. 
More generally, one has~\cite{vbc, Ivanov, FM:Grom, LS,
  Ivanov:covers}:

\begin{thm}[Gromov's vanishing theorem]\label{grom:van:thm}
  Let $X$ be a topological space. Then: 
\begin{enumerate}
\item the map~$\comp_X^s \colon H^s_b(X) \to H^s(X)$ is zero for all $s \geq \amcat(X)$;
\item we have~$\sv{\alpha}_1 = 0$ for all $\alpha \in H_s(X; \R)$ with~$s \geq \amcat(X)$.
\end{enumerate}  
\end{thm}

\begin{rem}
  Usually, the vanishing theorem is stated in terms of multiplicity of
  the cover instead of cardinality. For CW-complexes, the two
  formulations are indeed equivalent~\cite[Remark~3.13]{clm}.
\end{rem}

\subsection{Vanishing results for fibrations}\label{subsec:simvolfib} 

We apply Theorem~\ref{thm:main} and
Theorem~\ref{rem:equal:cat:cat:amenable} in the case of the
class~$\Am$ of amenable groups to obtain vanishing results for the
comparison map and simplicial volume.  The class~$\Am$ is closed under
subgroups, isomorphisms, quotients, and extension by Abelian groups.

\begin{cor}[vanishing result for fibrations]\label{cor:compfib}
  Let $p\colon E \longrightarrow B$ be a fibration with a
  path-connected base space. Let $x_0 \in B$ be a non-degenerate
  basepoint of~$B$ and let $F := p^{-1}(x_0)$. Moreover, let
  $s \geq \amcat (F) \cdot \lscat(p)$. Then:
  \begin{enumerate}
  \item the comparison map
    $
    \comp_E^s \colon H^s_b(E) \to H^s(E)
    $  is zero;
  \item all classes $\alpha \in \, H_s(E; \R)$ have vanishing $\ell^1$-seminorm $\sv{\alpha}_1 = 0$.
  \end{enumerate}
\end{cor}
\begin{proof}
  From Theorem~\ref{thm:G:cover:fibrations} and
  Remark~\ref{rem:equal:cat:cat:amenable} we obtain $\amcat(E) \leq
  \amcat(F) \cdot \lscat(p)$. Thus, the claim follows by applying
  Gromov's vanishing theorem~\ref{grom:van:thm}.
\end{proof}

\begin{cor}[simplicial volume and fibre bundles]\label{cor:G:cover:fibrations:vanishing:simvol}
  Let $M$ be an oriented closed connected manifold that is
  the total space of a fibre bundle~$p\colon M \to B$ with oriented
  closed connected fibre~$N$ and base~$B$. If
  \[ \amcat (N) \leq \frac{\dim (M)}{\dim(B) + 1},
  \]
  then~$\|M\| = 0$.
\end{cor}
\begin{proof}
  In view of Gromov's vanishing theorem (Theorem~\ref{grom:van:thm}), it
  suffices to show that $\amcat(M) \leq \dim(M)$. 
  Using Theorem~\ref{thm:main}, the fact that $\lscat(B) \leq \dim(B) +1$,
  and the hypothesis on~$N$, we indeed obtain
  \begin{align*}
    \amcat(M)
    &
    \leq \amcat(N) \cdot  \bigl(\dim(B) +1 \bigr)
    \\
    &
    \leq \dim(M),
  \end{align*}
  as desired.
\end{proof}

\begin{cor}[simplicial volume of mapping tori]\label{cor:vanishing:simvol:mapping:tori}
  Let $M$ be an oriented closed connected manifold that is a mapping
  torus of a self-homeomorphism of an oriented closed connected
  manifold~$N$ with
  \[2 \cdot  \amcat(N) \leq \dim(N) + 1 \ .
  \] 
  Then, we have 
  $
  \sv{M} = 0.
  $
\end{cor}
\begin{proof}
  This is a special case of Corollary~\ref{cor:G:cover:fibrations:vanishing:simvol}.
\end{proof}

\begin{example}
  A classical question in hyperbolic geometry is to understand when
  hyperbolic manifolds fiber over the circle~\cite{AgolVirtual, Agol}.  Since the Euler
  characteristic in even dimension is proportional to the volume of
  hyperbolic manifolds, it is immediate to see that there are no even
  dimensional hyperbolic manifolds that fiber over the circle. On the
  other hand, the question is still open in odd dimension greater than~$3$.

  As hyperbolic manifolds have non-zero simplicial volume
  (Remark~\ref{rem:possimvol}),
  Corollary~\ref{cor:vanishing:simvol:mapping:tori} shows at least that
  odd-dimensional hyperbolic manifolds that fiber over the circle
  cannot have fiber with \emph{small} amenable category.
\end{example}

\begin{rem}\label{rem:bn}
  Corollary~\ref{cor:vanishing:simvol:mapping:tori} shows that mapping
  tori over connected sums~$M$ of amenable manifolds (of dimension at
  least~$3$) have zero simplicial volume (because~$\amcat(M) \leq
  \dim(M)$~\cite[Lemma~1]{GGH}\cite[Proposition~6.7]{clm}). This
  result may be interpreted as an extension of the classical result
  about the vanishing of the simplicial volume of manifold fibre
  bundles with amenable fiber~\cite[Exercise~14.15]{lueckl2}.

  Bucher and Neofytidis established vanishing results for simplicial
  volume of certain mapping tori over connected sums~$M$ of manifolds
  with zero simplicial
  volume~\cite[Theorem~1.7]{bucherneofytidis}. Their approach uses
  refined information on the structure of the self-glueing
  homeomorphism~$M \to M$.  Not all of these vanishing results can be
  recovered from Corollary~\ref{cor:vanishing:simvol:mapping:tori}
  (which is ignorant of the glueing map) and vice versa.
\end{rem}

\begin{rem}
  In the situation of mapping tori, the open amenable covers obtained
  via the proof of Corollary~\ref{cor:vanishing:simvol:mapping:tori}
  coincide with the obvious one obtained by \emph{doubling} an optimal
  open amenable cover~$(U_i)_{i\in I}$ of the fiber: We can ``split''
  the mapping torus of~$f \colon N \to N$ into two open overlapping
  cylinders~$N \times J_1$ and $N\times J_2$ that are glued
  appropriately. Then the mapping torus bundle is trivial over $J_1$ and
  $J_2$ and $(U_i \times J_1)_{i \in I} \cup (U_i \times
  J_2)_{i \in I}$ gives an amenable open cover of the mapping torus
  of~$f$ consisting of~$2 \cdot |I|$ elements.
\end{rem}

\begin{rem}
  It is tempting to prove Corollary~\ref{cor:compfib},
  Corollary~\ref{cor:G:cover:fibrations:vanishing:simvol}, and
  Corollary~\ref{cor:vanishing:simvol:mapping:tori} via the
  Hochschild-Serre spectral sequence for (bounded) cohomology.
  However, as there is no ``five lemma for zero maps'', there does not
  seem to be a direct way to do this.
\end{rem}

\begin{example}\label{rem:genmain}
  The following example shows that in Theorem~\ref{thm:main}, we
  cannot replace~$\lscat(B)$ by~$\catop_\mathscr{G} (B)$: There exist mapping
  tori~$M$ of oriented closed connected hyperbolic surfaces~$N$ that
  are oriented closed connected hyperbolic $3$-manifolds. Because $\amcat(S^1) = 1$,
  we then have (Remark~\ref{rem:possimvol}, Theorem~\ref{grom:van:thm})
  \begin{itemize}
  \item $\amcat(N) \cdot \amcat (S^1) = 3 \cdot 1$, but
  \item $\amcat(M) = 4$.
  \end{itemize}
\end{example}

\section{The fibre (non-)collapsing assumption}\label{sec:fnca}

In the context of minimal volume entropy, growth conditions on groups
naturally occur. We will explain how the fibre collapsing and
non-collapsing conditions by Babenko and Sabourau are related 
to categorical invariants for classes of groups with controlled
growth. In Section~\ref{sec:appminvolent}, we will apply our estimates
for fibrations to this setting.

\subsection{Groups with controlled growth}

To state the fibre (non-)collapsing conditions, we introduce the
corresponding classes of groups with controlled growth. Because
categorical invariants work better with classes of groups that are closed
under subgroups, we consider the following construction:

\begin{rem}\label{rem:classcomplete}
  Let $\mathscr{G}$ be a class of groups. We set
  \[ \overline{\mathscr{G}} :=
  \bigr\{ \Gamma \in \Ob(\Group)
  \bigm| \fa{\Lambda \leq \Gamma} (\text{$\Lambda$ finitely generated} \Rightarrow \Lambda \in \mathscr{G})
  \bigl\}.
  \]
  Then we have:
  \begin{enumerate}
  \item The class~$\overline{\mathscr{G}}$ is closed under taking subgroups.
  \item If $\mathscr{G}$ is closed under isomorphisms, then $\overline{\mathscr{G}}$ is closed
    under isomorphisms.
  \item If $\mathscr{G}$ is closed under quotients, then $\overline{\mathscr{G}}$ is closed under
    quotients.
  \item If $\mathscr{G}$ is closed under taking finitely generated subgroups,
    then the finitely generated groups in~$\overline{\mathscr{G}}$ coincide with
    the finitely generated groups in~$\mathscr{G}$.
  \end{enumerate}
\end{rem}

\begin{example}[classes of groups of controlled growth]
  Let $\delta \in \R_{>0}$. 
  The standard inheritance properties of growth conditions in finitely
  generated groups show that the following classes of groups are closed
  with respect to isomorphisms, finitely generated subgroups, and quotients:
  \begin{itemize}
  \item The class~$\Polyfg$ of finitely generated groups of polynomial growth.
    By the polynomial growth theorem~\cite{gromovpoly}, this class coincides
    with the class of all finitely generated virtually nilpotent groups.
  \item The class~$\Subexpfg$ of finitely generated groups of subexponential growth.
  \item The class~$\Subexpfg_{<\delta}$ of finitely generated groups of subexponential growth
    with subexponential growth rate~$< \delta$.
  \end{itemize}
  By Remark~\ref{rem:classcomplete}, the associated classes
  \[ \Poly := \overline \Polyfg,
  \quad
  \Subexp := \overline \Subexpfg,
  \quad
  \Subexp_{<\delta} := \overline{\Subexpfg_{<\delta}}
  \]
  are closed under isomorphisms, subgroups, and quotients. Moreover,
  the finitely generated groups in these classes are exactly the groups
  in~$\Polyfg$, $\Subexpfg$, $\Subexpfg_{<\delta}$, respectively. 
  
  Let~$\Expfg_{<\delta}$ be the class of finitely generated groups
  that admit a finite generating set whose growth rate is at most
  exponential of exponential growth rate~$<\delta$. In other words,
  a finitely generated group~$\Gamma$ does \emph{not} lie in~$\Expfg_{<\delta}$
  if and only if its uniform exponential growth rate~$\uexp(\Gamma)$
  is at least~$\delta$. It should be noted that $\Expfg_{<\delta}$
  is \emph{not} closed under taking finitely generated subgroups. 
\end{example}

\subsection{The fibre (non-)collapsing assumption}

We recall the fibre \mbox{(non-)}\allowbreak collapsing assumptions by Babenko and
Sabourau~\cite{bsfibre}. For convenience, we formulate the collapsing
condition for classes of groups; geometrically relevant choices
are the classes~$\Poly$, $\Subexp$, and $\Subexp_{<\delta}$.

\begin{defi}[fibre collapsing assumption; FCA]
  Let $\mathscr{G}$ be a class of groups.
  \begin{itemize}
  \item
    Let $k \in \N$. 
    A finite simplicial complex~$X$ satisfies the \emph{fibre
      collapsing assumption with respect to~$\mathscr{G}$ in dimension~$k$} if there exists a
    simplicial map~$f \colon X \longrightarrow P$ to a finite
    simplicial complex~$P$ with $\dim P \leq k$ and such that for
    all points~$p \in P$ (not necessarily vertices), the fibre~$f^{-1}(p)$
    is a $\mathscr{G}$-subset of~$X$.
  \item
    A finite simplicial complex~$X$ satisfies the \emph{fibre
      collapsing assumption with respect to~$\mathscr{G}$} if
    $X$ satisfies the fibre collapsing assumption with respect to~$\mathscr{G}$
    in dimension~$\dim X - 1$. 
  \end{itemize}
\end{defi}

\begin{rem}[fundamental groups of fibres are finitely generated]\label{rem:fingenfib}
  Let $X$ be a finite simplicial complex, let $P$ be a simplicial
  complex, and let $f \colon X \longrightarrow P$ be a simplicial
  map. Then, for every point~$p \in P$, the fibre~$f^{-1}(p)$
  is a finite simplicial complex. 
  In particular, the fundamental groups of all components
  of~$f^{-1}(p)$ are finitely generated.
\end{rem}

\begin{rem}\label{rem:fcafg}
  Let $\mathscr{G}$ be a class of groups that is closed under
  finitely generated subgroups, let $X$ be a finite simplicial
  complex, and let $k \in \N$.  Then $X$ satisfies the fibre
  collapsing assumption with respect to~$\mathscr{G}$ in dimension~$k$ if and
  only if it does so for~$\overline{\mathscr{G}}$:

  As $\mathscr{G}$ is closed under finitely generated subgroups, we have~$\mathscr{G}
  \subset \overline{\mathscr{G}}$ (Remark~\ref{rem:classcomplete}). In
  particular, the FCA for~$\mathscr{G}$ implies the FCA for~$\overline{\mathscr{G}}$. For
  the other implication, we argue as follows: Let $P$ be a simplicial
  complex and let $f \colon X \longrightarrow P$ be a simplicial
  map. Then, for every point~$p \in P$, the fundamental groups of all
  components of~$f^{-1}(p)$ are finitely generated
  (Remark~\ref{rem:fingenfib}). As every finitely generated subgroup
  in~$\overline{\mathscr{G}}$ also lies in~$\mathscr{G}$, the complex $X$ satisfies the
  fibre collapsing assumption with respect to~$\mathscr{G}$ in dimension~$k$ if
  it does so for~$\overline{\mathscr{G}}$.
\end{rem}

\begin{defi}[fibre non-collapsing assumption; FNCA]
  A finite simplicial complex~$X$ satisfies the \emph{fibre
    non-collapsing assumption} if there exists a~$\delta \in \R_{>0}$
  with the following property: For each finite simplicial complex~$P$
  with~$\dim P < \dim X$ and for each simplicial map~$f \colon X
  \longrightarrow P$, there exists a point~$p \in P$ and
  an~$x \in f^{-1}(p)$ such that~$\pi_1(i)(\pi_1(f^{-1}(p), x))$ has
  uniformly exponential growth with exponential growth rate at
  least~$\delta$, where $i \colon f^{-1}(p) \hookrightarrow X$ denotes
  the inclusion map (in other words: $f^{-1}(p)$ is \emph{not}
  an~$\Expfg_{<\delta}$-subset of~$X$).
\end{defi}

\begin{prop}\label{prop:fncacov}
  Let $\pi \colon \overline X \longrightarrow X$ be a simplicial
  finite-sheeted covering of finite simplicial complexes. If
  $\overline X$ satisfies the fibre non-collapsing assumption, then so
  does~$X$.
\end{prop}
\begin{proof}
  Let $\delta \in \R_{>0}$ be such that $\overline X$ satisfies the FNCA
  with uniformly exponential growth rate~$\geq \delta$.
  We write~$d \in \N$ for the number of sheets of~$\pi$
  and show that $X$ satisfies the FNCA with uniformly exponential
  growth rate~$\geq \sqrt[2d-1]{\delta}$. The basic reason is that
  uniformly exponential growth is inherited by finite index supergroups.

  Let $P$ be a finite simplicial complex with~$\dim P < \dim X = \dim
  \overline X$ and let $f \colon X \longrightarrow P$ be a simplicial
  map. 
  Then, $\overline f := f \circ \pi \colon \overline X \longrightarrow P$
  is a simplicial map. As $\overline X$ satisfies FNCA, there is a point~$p \in P$
  and an~$\overline x \in \overline f^{-1}(p) \subset \overline X$ such that
  the image
  \[ \overline \Lambda := \pi_1(\overline i) \bigl(\pi_1(\overline f^{-1}(p),\overline x)\bigr)
  \subset \pi_1(\overline X , \overline x)
  \]
  has uniform exponential growth rate~$\geq \delta$, where $\overline i$
  denotes the inclusion into~$\overline X$.
  As $\pi_1(\pi)$ is injective, the group (where $i$ is the inclusion into~$X$)
  \[ \Gamma := \pi_1(i) \bigl(\pi_1(f^{-1}(p),x)\bigr) \subset \pi_1(X,x)
  \]
  contains a finite index subgroup~$\Lambda$ that
  is isomorphic to~$\overline \Lambda$ and has index at
  most~$[\pi_1(\overline X,\overline x) : \overline \Lambda] = d$.
  Therefore, $\Gamma$ has uniformly exponential growth with uniformly
  exponential growth rate at
  least~\cite[Proposition~3.3]{shalenwagreich}\cite[Proposition~2.4]{dlharpeuexp}
  \[ \sqrt[\protect{2 \cdot [\Gamma :\Lambda]-1}]{\uexp (\Lambda)}
    \geq \sqrt[2 \cdot d - 1]{\uexp (\overline \Lambda)}
    \geq \sqrt[2 \cdot d -1 ]{\delta}.
    \qedhere
  \]
\end{proof}

\subsection{F(N)CA via category invariants}

As observed by Babenko and Sabourau~\cite[Proposition~2.13,
  Proposition~3.10]{bsfibre}, the fibre collapsing and non-collapsing
assumptions are connected to multiplicity conditions on
open covers. We will recast this result in terms of categorical
invariants.

\begin{prop}[fibre conditions and categorical invariants]\label{prop:bscat}
  Let $\mathscr{G}$ be a class of groups that is closed under isomorphisms, let
  $X$ be a finite simplicial complex,
  and let $k \in \N$. Then:
  \begin{enumerate}
  \item If $X$ satisfies the fibre collapsing assumption with respect to~$\mathscr{G}$
    in dimension~$k$, then
    \[ \catop_\mathscr{G} (X) \leq k + 1.
    \]
  \item
    If in addition $\mathscr{G}$ is closed under finitely generated subgroups
    and if $\catop_\mathscr{G} (X) \leq k+1$, then there exists an iterated
    barycentric subdivision of~$X$ that satisfies the fibre collapsing
    assumption with respect to~$\mathscr{G}$ in dimension~$k$.
  \end{enumerate}
\end{prop}

\begin{proof}
  The proof of Babenko and Sabourau~\cite[Proposition~2.13]{bsfibre}
  for the connection between the fibre collapsing assumption and
  multiplicities of open covers with $\pi_1$-restrictions also works
  in the full generality of classes of groups. The condition on
  the multiplicity of the open covers can be adapted into a 
  condition on~$\catop_\mathscr{G}$. 
  For the sake of completeness, we recall the arguments:

  \emph{Ad~1.}
  Let $f \colon X \longrightarrow P$ be a simplicial map witnessing 
  that $X$ satisfies the fibre collapsing assumption with respect to~$\mathscr{G}$
  in dimension~$k$.
  Taking the barycentric subdivision, yields a simplicial map~$f' \colon X'
  \longrightarrow P'$ between the barycentric subdivisions that witnesses
  that $X'$ satisfies the fibre collapsing assumption with respect to~$\mathscr{G}$
  in dimension~$k$ (as subsets of the geometric realisations, the fibres
  of~$f$ and~$f'$ agree). Because $X'$ is homeomorphic to~$X$, it suffices
  to show that $\catop_\mathscr{G} (X') \leq k +1$.

  Let $\mathcal{U} = (U_i)_{i \in I}$ be the open stars cover of~$P'$, regrouped
  and indexed by the dimensions of the underlying simplices in~$P$.
  Then $|\mathcal{U}| \leq \dim P + 1 = k +1$. 
  We now consider the pull-back cover~$\mathcal{V} := (V_i)_{i \in I}$ of~$X'$, where
  $V_i := f'{}^{-1}(U_i)$ for all~$i \in I$. Then $\mathcal{V}$ is an open cover of~$X'$
  with~$|\mathcal{V}| \leq |U| \leq k+1$. It thus suffices to show that each~$V_i$
  is a $\mathscr{G}$-subset of~$X'$.

  Let $i \in I$. Because $f'$ is a simplicial map, there exists a
  vertex~$p_i \in P'$ such that $V_i = f'{}^{-1}(U_i)$ deformation
  retracts onto the fibre~$f'{}^{-1}(p_i)$. Let $j_i \colon f'{}^{-1}(p_i)
  \hookrightarrow V_i$ and $k_i \colon V_i \hookrightarrow X'$ denote
  the inclusions. Then $j_i$ is a homotopy equivalence and so
  \[ \pi_1(k_i)\bigl(\pi_1(V_i,x)\bigr)
     \cong \pi_1(k_i \circ j_i) \bigl(\pi_1(f'{}^{-1}(p_i), x)\bigr)
  \]
  for all~$x \in f'{}^{-1}(p_i)$. Therefore, $V_i$ is a $\mathscr{G}$-subset of~$X'$
  and we conclude that $\catop_\mathscr{G}(X) = \catop_\mathscr{G}(X') \leq k+1$. 
  
  \emph{Ad~2.}
  For the converse implication, we use the nerve construction.  Let
  $\mathcal{U}$ be an open $\mathscr{G}$-cover of~$X$ with~$|\mathcal{U}| \leq k
  +1$.  Then, the nerve~$P$ of~$\mathcal{U}$ is a finite simplicial
  complex with~$\dim P = \mult \mathcal{U} - 1 \leq k$.  Let $\Phi$ be
  a partition of unity subordinate to~$\mathcal{U}$ and let $f \colon X
  \longrightarrow P$ be the nerve map associated with~$\Phi$. In
  general, $f$ is not simplicial; this can be handled as follows: By
  the Lebesgue lemma, there is an iterated barycentric
  subdivision~$X'$ of~$X$ such that each simplex of~$X'$ is contained
  in one of the sets in~$\mathcal{U}$ and such that $f$ admits a simplicial
  approximation~$f' \colon X' \longrightarrow P$.  If $p \in P$, then
  $f'{}^{-1}(p)$ is contained in one of the elements~$U_i$ of~$\mathcal{U}$.  In
  particular: If $j_i \colon f'{}^{-1}(p) \hookrightarrow U_i$ and
  $k_i \colon U_i \longrightarrow X'$ denote the inclusions, then
  \[ \pi_1(k_i \circ j_i) \bigl(\pi_1(f'{}^{-1}(p),x) \bigr)
    \subset \pi_1(k_i) \bigl(\pi_1(U_i,x)\bigr) 
  \]
  holds for all~$x \in f'{}^{-1}(p)$. Because $U_i$ is a $\mathscr{G}$-subset
  of~$X$ (whence~$X'$), because $\mathscr{G}$ is closed under
  finitely generated subgroups, and because $f'$ is a simplicial map,
  also $f'{}^{-1}(p)$ is a $\mathscr{G}$-subset of~$X'$ (Remark~\ref{rem:fingenfib}).
\end{proof}

\begin{cor}[FCA via cat]\label{cor:fcacat}
  Let $X$ be a finite simplicial complex, let $k \in \N$, and let $\mathscr{G}$
  be a class of groups that is closed under isomorphisms and finitely
  generated subgroups. Then the following are equivalent:
  \begin{enumerate}
  \item There exists an iterated barycentric subdivision~$X'$ of~$X$
    that satisfies the fibre collapsing assumption with respect to~$\mathscr{G}$
    in dimension~$k$.
  \item We have $\catop_\mathscr{G}(X) \leq k + 1$.
  \item We have $\catop_{\overline{\mathscr{G}}} (X) \leq k + 1$.
  \end{enumerate}
\end{cor}
\begin{proof}
  \emph{Ad~$1. \Longrightarrow 2.$}
  Let us suppose that there exists an iterated barycentric
  subdivision~$X'$ of~$X$ that satisfies the fibre collapsing
  assumption with respect to~$\mathscr{G}$ in dimension~$k$.  Because $\mathscr{G}$ is
  closed under finitely generated subgroups,
  Proposition~\ref{prop:bscat} shows that
  \[ \catop_\mathscr{G} (X') \leq k+1.
  \]
  As $X'$ and $X$ are homeomorphic, we obtain~$\catop_\mathscr{G} (X) =
  \catop_\mathscr{G} (X') \leq k+1$.

  \emph{Ad~$2. \Longrightarrow 3.$}
  This is a direct consequence of the fact that~$\mathscr{G} \subset \overline{\mathscr{G}}$.
  
  \emph{Ad~$3. \Longrightarrow 1.$}
  Let $\catop_{\overline{\mathscr{G}}} (X) \leq k+1$. As $\overline{\mathscr{G}}$
  is closed under subgroups, by Proposition~\ref{prop:bscat}, there
  exists an iterated barycentric subdivision~$X'$ of~$X$ that
  satisfies the fibre collapsing assumption with respect to~$\overline{\mathscr{G}}$ in dimension~$k$. We can now apply Remark~\ref{rem:fcafg} to pass
  to~$\mathscr{G}$.
\end{proof}

\begin{example}
  Let $X$ be a finite simplicial complex. Then, by Corollary~\ref{cor:fcacat}, 
  $\catop_{\Poly} (X) \leq \dim X$ is equivalent to the existence of an
  iterated barycentric subdivision of~$X$ that satisfies the fibre
  collapsing condition with polynomial growth.
\end{example}

\begin{rem}[dimension~$2$]\label{rem:faccat2}
  The following generalisation of a result of Bregman and
  Clay~\cite[Proposition~4.1]{bregmanclay} is an instance of general
  considerations on categorical invariants: Let $\mathscr{G}$ be a class of
  groups that is closed under isomorphisms, finitely generated
  subgroups, and quotients. 
  Let $\Gamma$ be a group that does \emph{not} lie in~$\mathscr{G}$ and
  let $X$ be a finite simplicial complex with~$\pi_1(X) \cong \Gamma$.
  Then the following are equivalent~\cite[Corollary~5.4 and the
    subsequent remark]{clm}:
  \begin{enumerate}
  \item The group~$\Gamma$ is the fundamental group of a graph of
    groups whose vertex and edge groups all lie in~$\overline{\mathscr{G}}$.
  \item We have~$\catop_{\overline{\mathscr{G}}} (X) = 2$.
  \end{enumerate}
  If $X$ is of dimension~$2$, by 
  Corollary~\ref{cor:fcacat}, these conditions are equivalent to:
  \begin{enumerate}
  \setcounter{enumi}{2}
  \item There exists an iterated barycentric subdivision of~$X$ that
    satisfies the fibre collapsing assumption with respect to~$\mathscr{G}$.
  \item We have~$\catop_\mathscr{G} (X) = 2$.
  \end{enumerate}
  For example, this applies to the classes~$\Poly$, $\Subexp$, and
  $\Subexp_{<\delta}$. 
\end{rem}

\begin{cor}[FNCA via cat]\label{cor:fncacat}
  Let $X$ be a finite simplicial complex and let $\delta \in \R_{>0}$.
  If $\catop_{\Expfg_{<\delta}}(X) > \dim X$, then $X$ satisfies the fibre
  non-collapsing condition (with uniform exponential growth rate~$\delta$).
\end{cor}
\begin{proof}
  This follows from the definition of the fibre non-collapsing
  condition and the contraposition of the first part of
  Proposition~\ref{prop:bscat}.
\end{proof}

It is not clear to us that the converse of Corollary~\ref{cor:fncacat}
also holds (up to subdivision) because $\Expfg_{<\delta}$ is not closed
under finitely generated subgroups.

For later use, we give an example that slightly generalises an example
by Babenko and Sabourau~\cite[Proposition~3.7]{bsfibre}:

\begin{example}\label{exa:fncaprod}
  Let $N_1, \dots, N_r$ be oriented closed connected rationally
  essential smooth manifolds of dimension~$\geq 2$ with non-elementary
  hyperbolic fundamental group.  Then the product~$M := N_1 \times
  \dots \times N_r$ satisfies the FNCA with respect to every
  triangulation: We proceed in the following steps:
  \begin{enumerate}
    \setcounter{enumi}{-1}
  \item If $\Gamma$ is a finitely generated hyperbolic
    group, then there exists a~$\delta_\Gamma \in \R_{>0}$ such that:
    Every finitely generated subgroup~$\Lambda$ of~$\Gamma$ is virtually
    cyclic or satisfies~$\uexp(\Lambda) \geq \delta_\Gamma$.
  \item There exists a~$\delta \in \R_{>0}$ such that: Every finitely
    generated subgroup~$\Lambda$ of~$\pi_1(M)$ is amenable or
    satisfies~$\uexp(\Lambda) \geq \delta$.
  \item $\amcat (M) > \dim (M)$.
  \item $\catop_{\Exp_{< \delta}} (M) > \dim (M)$. 
  \end{enumerate}
  The last property implies that $M$ satisfies the FNCA
  (Corollary~\ref{cor:fncacat}).

  \emph{Ad~0.} This is a result of Delzant an
  Steenbock~\cite[Theorem~1.1]{delzantsteenbock}.
  
  \emph{Ad~1.}
  We apply part~0 to the~$\pi_1(N_j)$ and set
  $\delta := \min (\delta_{\pi_1(N_1)}, \dots, \delta_{\pi_1(N_r)})$.
  Let $\Lambda \subset \pi_1(M)$ be a finitely generated subgroup.
  We distinguish two cases:
  \begin{itemize}
  \item For all~$j \in \{1,\dots, r\}$, the projection~$p_j(\Lambda)
    \subset \pi_1(N_j)$ is virtually cyclic.
  \item There exists a~$j \in \{1,\dots, r\}$ such that $p_j(\Lambda)
    \subset \pi_1(N_j)$ is not virtually cyclic. 
  \end{itemize}
  In the first case, $\Lambda$ is isomorphic to a subgroup of a product
  of $r$ virtually cyclic groups and
  thus amenable. In the second case, $\Lambda$ projects onto a subgroup~$\Lambda_j$
  of the non-elementary hyperbolic group~$\pi_1(N_j)$ that
  is not virtually cyclic; thus,
  \[ \uexp(\Lambda) \geq \uexp(\Lambda_j) \geq \delta_{\pi_1(N_j)} \geq \delta.
  \]

  \emph{Ad~2.}
  Let $ j\in \{1,\dots, r\}$. Then $N_j$ has non-zero
  simplicial volume (Remark~\ref{rem:possimvol}). Therefore, also $M
  = N_1 \times \dots \times N_r$ has non-zero simplicial
  volume~\cite{vbc}. In particular, $\amcat (M) > \dim (M)$
  (Theorem~\ref{grom:van:thm}).

  \emph{Ad~3.} 
  By the first part, all finitely generated subgroups of~$\pi_1(M)$
  that lie in~$\Exp_{<\delta}$ are amenable. In combination with the
  second part, we obtain 
  $\catop_{\Exp_{<\delta}} (M) \geq \amcat (M) > \dim (M)$.
\end{example}
 
\section{Applications to minimal volume entropy}\label{sec:appminvolent}

We recall the definition of minimal volume entropy
(Section~\ref{subsec:minvolent}) and the (non-)vanishing results of
Babenko and Sabourau (Section~\ref{subsec:entnonvan}).  In
Section~\ref{subsec:entfibvan}, we derive the vanishing results for
minimal volume entropy and fibrations. In
Section~\ref{subsec:entsimvol}, we extend the FNCA examples of Babenko
and Sabourau.

\subsection{Minimal volume entropy}\label{subsec:minvolent}

The minimal volume entropy measures the minimal possible growth rate of
balls.

\begin{defi}[minimal volume entropy]
  Let $X$ be a finite connected simplicial complex.
  \begin{itemize}
  \item A \emph{piecewise Riemannian metric} on~$X$ is a family
    of Riemannian metrics on all simplices of~$X$ that is compatible
    along common sub-simplices. Let $\Riem(X)$ be the set of
    all piecewise Riemannian metrics on~$X$.
  \item
    Let $g$ be a piecewise Riemannian metric on~$X$. Then the
    \emph{volume entropy of~$(X,g)$} is defined as
    \[ \volent(X,g) := \lim_{R \to \infty} \frac1 R \cdot \log \vol \bigl( B(R,\widetilde x), \widetilde g\bigr),
    \]
    where $\widetilde x$ is a vertex of the universal covering~$\widetilde X$ of~$X$,
    where $\widetilde g$ is the pull-back of~$g$ to~$\widetilde X$, and where 
    $B(R,\widetilde x)$ is the ball of
    radius~$R$ around~$\widetilde x$ in~$\widetilde X$.
    (The choice of~$\widetilde x$ does not matter).
  \item The \emph{minimal volume entropy of~$X$} is defined by
    \[ \ent(X) := \inf_{g \in \Riem(X)} {\volent(X,g)} \cdot {\vol(X,g)^{1/\dim(X)}}.
    \]
  \end{itemize}
\end{defi}

\begin{rem}[minimal volume entropy and barycentric subdivisions]\label{rem:entbary}
  Let $X$ be a finite connected simplicial complex. 

  If $X'$ is the barycentric subdivision of~$X$, then $\ent(X') = \ent(X)$,
  as can be seen by smooth approximation of piecewise Riemannian metrics on~$X'$
  by piecewise Riemannian metrics on~$X$.

  Inductively, we obtain that if $X'$ is an iterated barycentric subdivision of~$X$,
  then $\ent(X') = \ent(X)$.
\end{rem}

\begin{rem}[minimal volume entropy of smooth manifolds]\label{rem:entmfd}
  Let $M$ be a closed connected smooth manifold. Then the \emph{minimal volume
    entropy of~$M$} is defined as
  \[ \ent(M) := \inf_{g \in \Riem(M)} \volent(M,g) \cdot \vol(M,g)^{1/\dim(M)},
  \]
  where $\Riem(M)$ is the set of all actual Riemannian metrics
  on~$M$. If $X$ is a finite simplicial complex that triangulates~$M$,
  then $\ent(M) = \ent(X)$~\cite[Lemma~2.3]{babenko}. In fact, minimal
  volume entropy is a topological invariant of smooth manifolds in a
  very strong sense~\cite{babenkosurgery,brunnbauer}.
\end{rem}

\subsection{The (non-)vanishing theorems by Babenko and Sabourau}\label{subsec:entnonvan}

\begin{thm}[FCA and vanishing; \protect{\cite[Theorem~1.3]{bsfibre}}]\label{thm:entfca}
  Let $X$ be a finite connected simplicial complex of dimension~$n$. If there
  is a $k \in \{0,\dots, n -1\}$ such that $X$ satisfies the fibre
  collapsing assumption for~$\Subexp_{<(n-k)/n}$ in dimension~$k$, then
  \[ \ent(X) = 0.
  \]
\end{thm}

\begin{thm}[FNCA and non-vanishing; \protect{\cite[Theorem~1.5]{bsfibre}}]\label{thm:entfnca}
    Let $X$ be a finite connected simplicial complex that satisfies the fibre
    non-collapsing assumption. Then
    \[ \ent(X) > 0.
    \]
\end{thm}

We can reformulate these (non-)vanishing results in terms of generalised
categorical invariants:

\begin{cor}\label{cor:entcat}
  Let $X$ be a finite connected simplicial complex.
  \begin{enumerate}
  \item If there is a~$k \in \{0,\dots, \dim X-1\}$ with~$\catop_{\Subexp_{< (n-k)/n}} (X) \leq k+1$,
    then $\ent(X) = 0$.
  \item If there is a~$\delta \in \R_{>0}$ with~$\catop_{\Expfg_{<\delta}}\!\!(X) > \dim X$,
    then $\ent(X) > 0$.
  \end{enumerate}
\end{cor}
\begin{proof}
  \emph{Ad~1.}
  If $\catop_{\Subexp_{< (n-k)/n}} (X) \leq k+1$, then an interated
  subdivision~$X'$ of~$X$ satisfies the fibre collapsing
  assumption for~$\Subexp_{<(n-k)/n}$ in dimension~$k$ (Corollary~\ref{cor:fcacat}).  Thus,
  $\ent(X') = 0$, by the vanishing theorem (Theorem~\ref{thm:entfca}).
  We then use that $\ent(X) = \ent(X')$ (Remark~\ref{rem:entbary}).

  \emph{Ad~2.}  This is an immediate consequence of
  Corollary~\ref{cor:fncacat} and the non-vanishing
  theorem (Theorem~\ref{thm:entfnca}).
\end{proof}

\subsection{Vanishing results for fibrations}\label{subsec:entfibvan}

We now apply Theorem~\ref{thm:main} in the case of growth classes of groups
to obtain vanishing results for minimal volume entropy.

\begin{cor}[minimal volume entropy and fibrations]\label{cor:entfibrations}
  Let $p \colon E \to B$ be a simplicial fibration of finite connected
  simplicial complexes.  Let $x_0 \in B$ be a vertex and let $F :=
  p^{-1}(x_0)$. If
  \[ \catop_{\Subexp_{< 1/\dim(X)}}(F) \cdot \lscat (B) \leq \dim (X),
  \]
  then~$\ent(X) = 0$.
\end{cor}
\begin{proof}
  Under the given hypotheses, Theorem~\ref{thm:main} shows that
  \[ \catop_{\Subexp_{< 1/\dim(X)}} (X)
  \leq \catop_{\Subexp_{< 1/\dim(X)}}(F) \cdot \lscat (B)
  \leq \dim (X).
  \]
  Therefore, Corollary~\ref{cor:entcat} implies~$\ent(X) = 0$.
\end{proof}

\begin{cor}[minimal volume entropy and fibre bundles]\label{cor:G:cover:fibrations:vanishing:minvolent}
  Let $M$ be an oriented closed connected smooth manifold that is
  the total space of a fibre bundle~$M \to B$ with oriented
  closed connected smooth fibre~$N$ and base~$B$. If
  \[ \catop_{\Subexp_{<1/\dim(M)}} (N) \leq \frac{\dim (M)}{\dim(B) + 1},
  \]
  then~$\ent(M) = 0$.
\end{cor}
\begin{proof}
  Triangulating~$M$, $B$, and $N$ and subdividing often enough, we may
  assume that we have a simplicial fibration between finite simplicial
  complexes of the corresponding dimensions. Moreover, the notions of minimal
  volume entropy for smooth manifolds and their triangulations
  coincide (Remark~\ref{rem:entmfd}). Using the estimate $\lscat (B) \leq \dim (B) +1$,
  the result follows from Corollary~\ref{cor:entfibrations}.
\end{proof}
  
\begin{cor}[minimal volume entropy of mapping tori]
  Let $X$ be a finite connected simplicial complex of dimension~$n$
  that fibres as a fibre bundle over the circle, with (simplicial)
  fibre~$F$. If
  \[ 2 \cdot \catop_{\Subexp_{<1/n}}(F) \leq \dim  F + 1,
  \]
  then $\ent X = 0$.
\end{cor}
\begin{proof}
  This is a special case of Corollary~\ref{cor:entfibrations}.
\end{proof}

\subsection{FNCA, minimal volume entropy, and simplicial volume}\label{subsec:entsimvol}

For manifolds, minimal volume entropy is an upper bound for simplicial
volume (up to a dimension
constant)~\cite{vbc}\cite[Th\'eor\`eme~D]{bcg}; in particular,
Remark~\ref{rem:possimvol} thus leads to examples of positive minimal
volume entropy. Conversely, the following is an open
problem~\cite{bsfibre}:

\begin{quest}
  Let $M$ be an oriented closed connected smooth manifold
  with~$\| M \| = 0$. Does this imply that $\ent(M) = 0$\;?
\end{quest}

By now, we know that the previous question admits a positive 
answer in dimension 2~\cite{Katok} and for all oriented closed connected 
geometric manifolds in dimension 3~\cite{ErikaP} 
(whence all by Perelemann's proof of Thurston's geometrisation conjecture)
 and 4~\cite{SuarezSerrato}. Even though we do not expect this question to have a positive answer in
full generality; a particularly interesting special case to
study would be aspherical oriented closed connected manifolds.

As shown by Babenko and Sabourau, for finite simplicial complexes,
FNCA (and thus positive minimal volume entropy) does not necessarily
imply the non-vanishing of ``simplicial volume''
(interpreted appropriately)~\cite[Theorem~1.6]{bsfibre}. Using a
variation of their construction, we obtain aspherical examples
of this type:

\begin{prop}\label{prop:fncasimvol}
  Let $n \in \N_{\geq 2}$. Then, there exists a finite connected
  simplicial complex~$X$ with the following properties:
  \begin{enumerate}
  \item The space~$X$ is aspherical.
  \item The complex~$X$ satisfies the fibre non-collapsing assumption.
  \item We have~$\dim X = n$ and $H_n(X;\Z) \cong 0$.
  \end{enumerate}
\end{prop}
\begin{proof}
  As~$n \geq 2$, there exists a~$k \in \N$ with~$n = 2 \cdot (k+1)$
  or~$n = 3 + 2 \cdot k$. 
  Let $N$ be the product of $k$ oriented closed connected hyperbolic
  surfaces. In the first case, let $\Sigma$ be a non-orientable
  closed connected hyperbolic surface; in the second case, we take
  a non-orientable closed connected hyperbolic $3$-manifold.
  We then set
  \[ M := \Sigma \times N
  \]
  and consider the orientation double covering~$p \colon \overline M
  \longrightarrow M$ of~$M$. Moreover, we triangulate~$\Sigma \times
  N$ and take the induced triangulation of~$\overline M$.

  Then $\overline M$ satisfies~FNCA (Example~\ref{exa:fncaprod}).
  Therefore also $M$ satisfies~FNCA (Proposition~\ref{prop:fncacov}). 
  By construction, $M$ is aspherical, $n$-dimensional, and $H_n(M;\Z)
  \cong 0$ (as $M$ is non-orientable).
 
  Alternatively, one can also carry out the same argument when $M$ is
  a non-orientable closed connected hyperbolic $n$-manifold; such
  manifolds indeed
  exist~\cite[Theorem~1.2]{longreid}\cite[Section~4.2]{kolpakovrioloslavich}. The
  argument above has the advantage that it does not need existence
  theorems of such manifolds in higher dimensions.
\end{proof}

\begin{cor}\label{cor:entsimvol}
  Let $ n\in \N_{\geq 2}$ and let $c \in \R_{>0}$. Then, there exists
  a finite connected simplicial complex~$X$ with the following
  properties:
  \begin{enumerate}
  \item The space~$X$ is aspherical.
  \item We have~$\ent(X) > c$.
  \item We have~$\dim X = n$ and~$H_n(X;\Z) \cong \Z$ as well as
    \[ \fa{\alpha \in H_n(X;\R)} \|\alpha\|_1 = 0.
    \]
  \end{enumerate}
\end{cor}
\begin{proof}
  Let $Y$ be an aspherical finite simplicial complex of dimension~$n$
  as provided by Proposition~\ref{prop:fncasimvol}. In particular, $\ent(Y) > 0$
  because of~FNCA (Theorem~\ref{thm:entfnca}). 
  Taking the wedge of a large enough number~$m$ of copies of~$Y$
  results in a finite aspherical simplicial complex~$Z := \bigvee_m Y$ of
  dimension~$n$ with~\cite[Theorem~2.6]{bsseminorm}
  \[ \ent(Z) \geq m \cdot \ent(Y) > c
  \qand
  H_n(Z;\Z) \cong 0.
  \]
  Then
  $X := (S^1)^{\times n} \lor Z
  $
  is a finite aspherical simplicial complex of dimension~$n$
  with
  \[ \ent(X) \geq \ent\bigl((S^1)^{\times n}\bigr) + \ent(Z) > c
  \qand H_n(X;\Z) \cong \Z.
  \]
  The fundamental class of~$(S^1)^{\times n}$ pushes forward to
  a generator~$\alpha$ of~$H_n(X;\R)$. As the $n$-torus has
  simplicial volume~$0$, it follows that~$\|\alpha\|_1 = 0$.
\end{proof}

 
\bibliographystyle{abbrv}
\bibliography{svbib}

\end{document}